\documentclass[journal]{IEEEtran}

\usepackage[nocompress]{cite}
\usepackage{amsmath,amsthm,amsfonts,mathrsfs,amssymb, cases}
\usepackage{enumerate, paralist}
\usepackage{longtable,tabularx,booktabs}
\usepackage{graphicx} 
\usepackage{tikz}
\usepackage{extarrows}
\usepackage{arydshln}
\usepackage[flushleft]{threeparttable}
\allowdisplaybreaks[4]

\usepackage{xcolor}
\usepackage[colorlinks,allcolors=teal]{hyperref} 
\usepackage{url,doi}

\newcommand{\F}{\mathbb{F}}
\newcommand{\Tr}{\mathrm{Tr}}
\newcommand{\N}{\mathrm{N}}

\newtheorem{thm}{Theorem}
\newtheorem{cor}[thm]{Corollary}
\newtheorem{lem}[thm]{Lemma}

\newtheorem{rem}{Remark}

\begin{document}
\renewcommand\arraystretch{1.5}
\setlength{\itemsep}{0.5ex}

\title{On Inverses of Permutation Polynomials of \\ Small Degree over Finite Fields}

\author{Yanbin~Zheng, Qiang~Wang, and~Wenhong~Wei
\thanks{Y. Zheng is with the School of Computer Science and Technology,
         Dongguan University of Technology, China, and
         with the Guangxi Key Laboratory of Cryptography and Information Security,
         Guilin University of Electronic Technology, China,
         and with the School of Computer Science and Engineering,
         South China University of Technology, China,
         and also with the Peng Cheng Laboratory, Shenzhen, China
         (e-mail: zhengyanbin@guet.edu.cn).}
\thanks{Q. Wang is with the School of Mathematics and Statistics, Carleton University,          Ottawa, Canada (e-mail: wang@math.carleton.ca).}
\thanks{W. Wei is with the School of Computer Science and Technology,
         Dongguan University of Technology, Dongguan, China
         (e-mail: 2017164@dgut.edu.cn).}
\thanks{{\color{purple} 
Please refer to this paper as: Y.~Zheng, Q.~Wang, and W.~Wei. On inverses of permutation polynomials of small degree over finite fields. IEEE Trans. Inf. Theory, 66(2):914--922, Feb. 2020.}
}
}

\markboth{T\MakeLowercase{his paper has been published by} IEEE
\MakeLowercase{available at http://dx.doi.org/10.1109/TIT.2019.2939113.
}}{Zheng \MakeLowercase{\textit{et al.}}:
On Inverses of Permutation Polynomials of Small Degree over Finite Fields}
\maketitle

\begin{abstract}
Permutation polynomials (PPs) and their inverses have applications in cryptography,
coding theory and combinatorial design theory. In this paper, we make a brief summary of the inverses of PPs of finite fields, and give the inverses of all PPs of degree $\leq 6$ over finite fields $\mathbb{F}_{q}$ for all $q$ and the inverses of all PPs of degree $7$ over  $\mathbb{F}_{2^n}$. The explicit inverse of a class of fifth degree PPs is the main result, which is obtained by using Lucas' theorem, some congruences of binomial coefficients, and a known formula for the inverses of PPs of finite fields.
\end{abstract}


\IEEEpeerreviewmaketitle

\section{Introduction}

\IEEEPARstart{F}{or} a prime power $q$, let $\F_{q}$ denote the finite field~with~$q$ elements,
$\F_{q}^{*} =\F_{q}\setminus\{0\}$, and $\F_{q}[x]$ the ring of polynomials over~$\F_{q}$.
A polynomial $f \in \F_{q}[x]$ is called a permutation polynomial (PP) of $\F_{q}$
if it induces a bijection from $\F_{q}$ to itself. Hence for any PP $f$ of $\F_{q}$, there exists a polynomial $f^{-1} \in \F_{q}[x]$ such that $f^{-1}(f(c)) =c$
for each $c \in \F_{q}$ or equivalently $f^{-1}(f(x)) \equiv x \pmod{x^{q} -x}$,
and $f^{-1}$ is unique in the sense of reduction modulo $x^q -x$.
Here $f^{-1}$ is defined as the composition inverse of $f$ on $\F_{q}$,
and we simply call it the inverse of $f$.

PPs of finite fields have been extensively studied
for their applications in coding theory, combinatorial design theory, cryptography, etc.
For instance, some PPs of $\F_{2^m}$ were used in~\cite{Ding-code} to construct binary cyclic codes. The Dickson PPs of degree~5 of $\F_{3^m}$  were employed
in~\cite{Ding-comb} to construct new examples of skew Hadamard difference sets,
which are inequivalent to the classical Paley difference sets.
In block ciphers, a permutation is often used as an S-box to build the confusion layer
during the encryption process and the inverse is needed while decrypting the cipher.
PPs are useful in the construction of bent functions~\cite{Sihem16Bent1,Sihem16Bent2,Bent18},
which have optimal nonlinearity for offering resistance against
the fast correlation attack on stream ciphers and the linear attack on block ciphers.
PPs were employed in~\cite{Golomb96} to construct circular Costas arrays,
which are useful in sonar and radar communications.
PPs were also applied in the construction of check digit systems
that detect the most frequent errors~\cite{check, Winterhof2014}.

The study of PPs of finite fields has a long history.
In 1897, Dickson~\cite{Dickson1897} listed all {\bf normalized} PPs of degree
$\leq 5$ of $\F_{q}$ for all $q$, and classified all PPs of degree $6$ of $\F_{q}$ for odd $q$.
In 2010, the complete classification of PPs of degree $6$ and~$7$ of $\F_{2^n}$
was settled in~\cite{JLi10}. In recent years, a lot of progress has been made on the constructions of PPs of finite fields;
see for example~\cite{Hou32, KLi17-23, NLi17-tri, DWu17, Zha17-tri}
for permutation binomials and trinomials of the form $x^rh(x^{q-1})$ of $\F_{q^2}$,
see~\cite{LWang18,DZheng19} for PPs of the form $(x^{q} -x +c)^s +L(x)$ of $\F_{q^2}$,
see~\cite{LLi18, Zheng-DCC} for PPs of the form
$(ax^{q} +bx +c)^r\phi((ax^{q} +bx +c)^s) +ux^{q}+vx$ of $\F_{q^2}$,
see \cite{Charpin09, Charpin10, Charpin14, Cepak17} for PPs of the form $x^s+\gamma h(f(x))$,
see \cite{LiQSL19} for PPs with low boomerang uniformity.
For a detailed introduction to the developments on PPs,
we refer the reader to~\cite{Hou15, FF, HFF} and the references therein.

The problem of explicitly determining the inverses of these PPs is a more challenging problem.
In theory one could directly use the Lagrange interpolation formula,
but for large finite fields this becomes very inefficient.
In fact, there are few known classes of PPs whose inverses have been obtained explicitly.
It is also interesting to note that the explicit formulae of inverses of low degree PPs
have been neglected in the literature. This motivates us to give a short review of the progress in this topic and find explicit expressions of inverses of all classes of PPs of degree~$\leq 7$ in~\cite{Dickson1897,JLi10,PPs6}.

The rest of the paper is organized as follows.
Section~II gives a brief summary of the results concerning the inverses of PPs of finite fields.
In Section~III, we obtain the inverses~of all PPs of degree~$6$ of finite fields $\F_{q}$
for all $q$ and the inverses of all PPs of degree~$7$ of $\F_{2^n}$. For simplicity,
we only list in Table~\ref{T1} all {\bf normalized} PPs of degree $\leq 5$ and their inverses.
In particular, the inverse of PP $F(x)=x^5 -2ax^3 +a^2x$ of $\F_{5^n}$
is the main result of this paper; see Theorem~\ref{x5-1}.
Section~IV starts with a formula for the inverse of an arbitrary PP,
which was first presented in~\cite{MR-1}.
This formula provides all the coefficients of the inverse of a PP $f(x)$
by computing the coefficients of $x^{q-2}$
in $f(x)^{k} \pmod{x^{q}-x}$ for $1\leq i \leq q-2$.
Based on this method, we convert the problem of computing $F^{-1}(x)$
into the problem of finding the values of four classes of binomial coefficients.
Section~V gives the explicit values of these binomial coefficients by using
Lucas' theorem and several congruences of binomial coefficients modulo~5.

\begin{table*} [t]
\caption{All normalized PPs of degree $\leq 5$ and their inverses}
\centering
\begin{threeparttable}
\begin{tabular}{l l l l}
  \toprule[1pt]\label{T1}
  Normalized PPs of $\F_{q}$ & Inverses & $q$~~($n \geq 1$) & Reference \\
  \midrule
  $x$    & $x$        & any $q$ \\
  $x^2$  & $x^{q/2}$  & $q = 2^n$ \\
  $x^3$  & $x^{(aq-a+1)/3}$~~$(a \equiv 1-q \pmod 3)$
        & $q \not\equiv 1 \pmod 3$ & Thm~\ref{monomials}\\
  $x^3 -ax$ ($a$ not a square)
    & $\sum_{i=0}^{n-1}a^{-\frac{3^{i+1} -1}{2}}x^{3^i}$
    & $q = 3^n$   & \cite{CoulterH04, Wu-L-bi}\\
  $x^4$  &  $x^{q/4}$  & $q = 2^n$  & Thm~\ref{monomials}\\
  $x^4 \pm 3x$ & $\mp(x^4 -3x)$ & $q=7$ \\
  $x^4 +ax$ ($a$ not a cube)
    & $a^{\frac{q-1}{3}}(1+a^{\frac{q-1}{3}})^{-1}
    \sum_{i=0}^{n-1}a^{-\frac{4^{i+1}-1}{3}}x^{4^i}$
    & $q = 2^{2n}$
    &\cite{CoulterH04, Wu-L-bi}\\
  $x^4 +bx^2 +ax$ ($ab \neq 0, S_{n} + aS_{n-2}^2 =1$)\tnote{*}
    & $\sum_{i=0}^{n-1}\big(S_{n-2-i}^{2^{i+1}}+a^{1 -2^{i+1}}S_{i}\big) x^{2^i}$
    & $q = 2^n$    & Cor~\ref{L-tri}\\
  $x^5$ & $x^{(aq-a+1)/5}$~~$(a \equiv (1-q)^3 \pmod{5})$
        & $q \not\equiv 1 \pmod 5$ & Thm~\ref{monomials}\\
  $x^5 +ax$ ($a^2 =2$)  & $x^5 +ax$ & $q=9$\\
  $x^5 -ax$ ($a$ not a fourth power)
    & $a^{\frac{q-1}{4}}(1-a^{\frac{q-1}{4}})^{-1}
    \sum_{i=0}^{n-1}a^{-\frac{5^{i+1} -1}{4}}x^{5^i}$
    & $q = 5^n$  & \cite{CoulterH04, Wu-L-bi}\\
  $x^5 \pm 2x^2$ & $x^5 \mp 2x^2$ & $q=7$\\
  $x^5 +ax^3  +3a^2x$ ($a$ not a square)
    & $-a^2x^9 -ax^7 +4x^5 +4a^5x^3 -5a^4x$ & $q=13$ \\
  $x^5 + ax^3 +5^{-1}a^2x$ ($a \neq 0$)\tnote{\dag}
    & $\sum_{i=1}^{\lfloor m/2\rfloor}\frac{m}{m-i}\binom{m-i}{i}
    (5^{-1}a)^{5i}x^{m - 2i}$~~$(m =\frac{3q^2-2}{5})$
    & $q \equiv \pm 2 \pmod 5$
    & \cite[Lem 4.8]{KLi18-1}\\
  $x^5 -2ax^3 +a^2x$ ($a$ not a square)
    & {\color{blue}$\sum_{i=0}^{n-1}\sum_{j=0}^{n-1}
         2a^{\frac{q -5^{i+1} -5^{j+1} +1}{4}}
         x^{\frac{q +5^i +5^j -1}{2}}$}
    & $q = 5^n$  & Thm~\ref{x5-1} \\
  $x^5 +ax^3 \pm x^2 +3a^2x$ ($a$ not a square)
    & $x^5 \pm (2ax^4 - 2x^2) +a^2x^3 +ax$ & $q=7$ \\
  \bottomrule[1pt]
  \end{tabular}
  \begin{tablenotes}
        \footnotesize
    \item[*] In \cite[Table 7.1]{FF}, the PP $L(x) = x^4 +bx^2 +ax$
        (if its only root in $\F_{2^n}$ is~0) was listed.
        Since $x^4 +bx^2$ is not a PP of $\F_{2^n}$ for any $b \in \F_{2^n}^{*}$,
        we divide $L(x)$ into $x^4$, $x^4 +ax$ ($a$ not a cube),
        and $x^4 +bx^2 +ax$ ($ab \neq 0$, $S_{n} + aS_{n-2}^2 =1$) in Table~\ref{T1}.
        The motivation is  to give the explicit inverses of $x^4$ and $x^4 +ax$.
        The sequence $\{S_{i}\}$ is defined as follows: $S_{-1} =0$, $S_{0} =1$,
        $S_{i} = b^{2^{i-1}}S_{i-1} +a^{2^{i-1}}S_{i-2}$ for $1 \leq i \leq n$.
    \item[\dag] The normalized PPs $x^5 + ax^3 +5^{-1}a^2x$ ($a$ arbitrary) of $\F_{q}$ with $q \equiv \pm 2 \pmod 5$ appeared in \cite[Table 7.1]{FF}.
         Since the case $a = 0$ is a part of the normalized PP $x^5$ of $\F_{q}$ with
         $q \not\equiv 1 \pmod 5$, we impose a restriction $a \neq 0$.
    \item[\ddag] In the expressions of the inverses, $\lfloor m/2\rfloor$ denotes the largest integer $\leq m/2$.
  \end{tablenotes}
\end{threeparttable}
\end{table*}

\section{PPs and their inverses} \label{review}

We now give a brief summary of the results concerning the inverses of PPs of finite fields,
some of which will be used in the next section.


{\bf{Linear PPs}}. For $a \neq 0$, $b \in \F_{q}$, $ax +b$ is a PP of $\F_{q}$
and its inverse is  $a^{-1}(x-b)$.

{\bf{Monomials}}. For positive integer $n$, $x^n$ is a PP of $\F_{q}$
if and only if $\gcd(n ,q-1)=1$. In this case, the inverse is $x^m$,
where $mn \equiv 1 \pmod{q-1}$.
In particular, the inverses of $x^n$ on $\F_{2^t}$ for some APN exponents~$n$ were given explicitly in~\cite{KyuS2014}.

{\bf{Dickson PPs}}. The Dickson polynomial $D_{n}(x,a)$ of the first kind of degree $n$
with parameter $a \in \F_{q}$ is given as
\[
D_{n}(x,a) = \sum_{i=0}^{\lfloor n/2\rfloor}
             \frac{n}{n-i} \binom{n-i}{i}(-a)^i x^{n-2i},
\]
where $\lfloor n/2\rfloor$ denotes the largest integer $\leq n/2$.
It is known that $D_{n}(x,a)$ is a PP of $\F_q$ if and only if $\gcd(n ,q^2 -1)=1$.
Its inverse was determined in~\cite{KLi18-1} by the following lemma.
\begin{lem}[\!\!{\cite[Lemma~4.8]{KLi18-1}}] \label{dickson-1}
  Let $m$, $n$ be positive integers such that $mn \equiv 1 \pmod{q^2-1}$.
  Then the inverse of $D_{n}(x,a)$ on $\F_q$ is $D_{m}(x, a^n)$.
\end{lem}

{\bf{PPs of the form $\mathbf{x^rh(x^s)}$}}.
The first systematic study of PPs of $\F_{q}$ of the form $f(x)=x^{r}h(x^s)$ was made in~\cite{WL91},
where $q-1 =ds$, $1 \le r <s$ and $h  \in \F_{q}[x]$.
A~criterion for $f$ to be a PP of $\F_{q}$ was given in \cite{WL91}. Later on, several equivalent criteria were found in other papers; see for instance \cite{PL01,Wang07, Zieve09}. Essentially, it says that $f$ is a PP of $\F_{q}$ if and only if $\gcd(r, s)= 1$ and $x^rh(x)^{s}$ permutes
$U_d := \{1, \omega, \cdots, \omega^{d-1}\}$,
where $\omega$ is a primitive $d$-th root of unity of $\F_q$.
\cite[Theorem~1]{MR-1} characterized all the coefficients of the inverse
of $x^r g(x^s)^{d}$ on $\F_q$, where $\gcd(r, q-1)=1$.
This result was generalized in~\cite{Wang-1},
and the inverse of $f$ on $\F_q$ was given by
\[
f^{-1}(x)=\frac{1}{d}\sum_{i=0}^{d-1}\sum_{j=0}^{d-1}
         \omega^{i(t-jr)}\big(x/h(\omega^{i}) \big)^{\widetilde{r} +js},
\]
where $1\le \widetilde{r} <s$ and $r\widetilde{r}+st=1$.
This inverse was obtained later in~\cite{Zheng-2} by a piecewise method.
When $\gcd(r, q-1)=1$, the inverses of $f$ on $\F_q$ was given in~\cite{KLi18-1} by
\[
f^{-1}(x) = \big(x^{q-s}h(\ell(x^s))^{s-1}\big)^{r'}\ell(x^s),
\]
where $rr' \equiv 1 \pmod{q-1}$ and $\ell(x)$ is the inverse of $x^rh(x)^s$ on $U_d$. The method employed in~\cite{KLi18-1} is a multiplicative analogue of~\cite{TW-1} and~\cite{Wu-bil}.

{\bf{Linearized PPs}}.
Suppose $L(x) = \sum_{i=0}^{n-1} a_i x^{q^i} \in \F_{q^n}[x]$.
It is known that $L$ is a PP of $\F_{q^n}$
if and only if the associate Dickson matrix
\[
D_{L}  :=
\left(
  \begin{array}{cccc}
 a_{0}       & a_{1}     &  \cdots  & a_{n-1}      \\
 a_{n-1}^{q} & a_{0}^{q} &  \cdots  & a_{n-2}^{q}  \\
 \vdots      & \vdots    &  \vdots  & \vdots       \\
 a_{1}^{q^{n-1}}  & a_{2}^{q^{n-1}} &  \cdots  & a_{0}^{q^{n-1}}
  \end{array}
\right)
\]
is nonsingular~\cite[Page 362]{FF}.
In this case, the inverse was given in~\cite[Theorem 4.8]{Wu-L} by
\[
L^{-1}(x) = (\det(D_{L}))^{-1}
\textstyle\sum_{i=0}^{n-1} \bar{a}_i x^{q^i},
\]
where $\bar{a}_i$ is the $(i,0)$-th cofactor of $D_{L}$, i.e.,
the determinant of $D_{L}$ is
$
\det(D_{L}) = a_0 \bar{a}_0 + \sum_{i=1}^{n-1} a_{n-i}^{q^i} \bar{a}_i.
$
The inverses of some special linearized PPs were also obtained;
see~\cite{Wu-L-bi} for the inverse of arbitrary linearized permutation binomial,
see~\cite{Wu-odd,Wu-PhD} for the inverse of $x +x^2 +\Tr(x/a)$ on $\F_{2^n}$.
Very recently, linearized PPs of the form $L(x) +K(x)$ of $\F_{q^n}$
and their inverses are presented in~\cite[Theorem 3.1]{Reis18},
where $L$ is a linearized PP of $\F_{q^n}$ and
$K$ is a nilpotent linearized polynomial such that $L\circ K=K\circ L$.

{\bf{Bilinear PPs}}.
The product of two linear functions is a bilinear function. Let $q$ be even and $n$ be odd.
The inverse of bilinear PPs $x(\Tr_{q^n/q}(x)+ax)$ of $\F_{q^n}$ was obtained in~\cite{Coulter-1}, where $a \in \F_{q}\setminus \F_{2}$. The inverse of more general bilinear PPs
\[
f(x) = x(L(\Tr_{q^n/q}(x)) +a\Tr_{q^n/q}(x) +ax)
\]
of $\F_{q^n}$ was given in~\cite{Wu-bil} in terms of the inverse of bilinear PP $xL(x)$ when restricted to $\F_{q}$, where $a \in \F_{q}^{*}$ and $L \in \F_{q}[x]$ is a 2-polynomial.

{\bf{PPs of the form $\mathbf{x^s+\gamma h(f(x))}$}}.
Let $\gamma \in \F_{q^n}^{*}$ be a $b$-linear translator with respect to $\F_q$ for the mapping
$f: \F_{q^n} \rightarrow \F_{q}$, i.e., $f(x +u\gamma) -f(x) =ub$
holds for all $x \in \F_{q^n}$, all $u \in \F_{q}$ and a fixed $b \in \F_{q}$.
\cite[Theorem 8]{Kyur11} stated that $F_{1}(x) = x +\gamma f(x)$
is a PP of $\F_{q^n}$ if $b \neq -1$ (it is actually also a necessary condition).
Its inverse was given in \cite[Theorem~3]{Kyur11} by
$F_{1}^{-1}(x) = x -\frac{\gamma}{b+1} f(x)$.
Let $h$ be an arbitrary mapping from $\F_{q}$ to itself.
\cite[Theorem 6]{Kyur11} stated that $F_{2}(x) = x +\gamma h(f(x))$ permutes $\F_{q^n}$
if and only if $u +bh(u)$ permutes $\F_{q}$.
When $b =0$, the inverse was given in \cite[Proposition 4]{Cepak17}
by $F_{2}^{-1}(x) = x +(p-1) \gamma h(f(x))$,
where $p$ is the characteristic of $\F_{q^n}$.
PPs of the form $F_{3}(x) = x^s +\alpha \Tr (x^t)$ of $\F_{2^n}$ were studied in~\cite{Charpin09,Charpin10,Charpin14},
where $1 \leq s$, $t \leq 2^n -2$, $\alpha \in \F_{2^n}^{*}$, and $\Tr$ is the absolute trace function. A criterion for $F_{3}$ to be a PP of $\F_{2^n}$ was given in \cite{Charpin10, Charpin14}. If $F_{3}$ is a PP of $\F_{2^n}$ and $t = s(2^i +1)$ for some $0 \leq i \leq n-1$ and $i \neq n/2$, then the inverse is given in \cite[Theorem 4]{Charpin14} by $F_{3}^{-1}(x) = (x +\alpha \Tr (x^{2^i +1}))^{r}$, where $r$ is the inverse of $s$ modulo $2^n -1$.

{\bf{Involutions}}.
An involution is a permutation such that its inverse is itself.
A systematic study of involutions over $\F_{2^n}$ was made in~\cite{Involutions}.
The authors characterized the involution property of monomials,
Dickson polynomials~\cite{Dick-Invo} and linearized polynomials over $\F_{2^n}$,
and proposed several methods of constructing new involutions from known ones.
In particular, involutions of the form $G(x)+\gamma f(x)$ were studied in~\cite{Involutions},
where $G$ is an involution, $\gamma \in \F_{2^n}^{*}$ and $f \in \F_{2^n}[x]$.
Involutions of the form $x^{r}h(x^{s})$ were studied in~\cite{ZhengYL+2019}.
Moreover, the number of fixed points of involutions over $\F_{2^n}$ was also discussed in~\cite{Involutions}. A class of involutions over $\F_{2^n}$ with no fixed points was given in~\cite{Reis18}. Involutions satisfying special properties were presented in~\cite{Sihem16Bent1,Sihem16Bent2,Bent18} to construct Bent functions.

{\bf{PPs from the AGW criterion}}.
The Akbary--Ghioca--Wang (AGW) criterion~\cite{AGW} is an important method for constructing PPs.
A necessary and sufficient condition for $f(x) = h(\psi(x))\varphi(x) +g(\psi(x))$
to be a PP of $\F_{q^n}$ was given in~\cite{AGW} by using the additive analogue of AGW criterion,
where $h, \psi, \varphi, g \in \F_{q^n}[x]$ satisfy some conditions.
In~\cite{TW-1}, the inverse of $f$ was written in terms of the inverses of two other polynomials
bijecting two subspaces of $\F_{q^n}$.
In some cases, these inverses can be explicitly obtained. Further extensions of~\cite{TW-1} can be found in~\cite{TW17}. The general results in \cite{TW-1, TW17} contain some concrete classes mentioned earlier such as bilinear PPs \cite{Wu-bil}, linearized PPs of the form $L(x)+K(x)$\cite{Reis18}, and PPs of the form $x+ \gamma f(x)$ with $b$-linear translator~$\gamma$ \cite{Kyur11}.

{\bf{Generalized cyclotomic mapping PPs}}.
Cyclotomic mapping PPs of finite fields were introduced in~\cite{Nied-cyc, Wang07},
and were generalized in~\cite{Wang-cyc}. A simple class of generalized cyclotomic mapping PPs
of $\F_q$ was defined in~\cite{Wang-cyc} as
\begin{equation}\label{cmpp}
f(x) 
=\frac{1}{d}\sum_{i=0}^{d-1}\sum_{j=0}^{d-1}a_{i}\omega^{-ij}x^{r_{i} +js},
\end{equation}
where $q-1=ds$, $a_i \in \F_{q}^{*}$, $1 \le r_i <s$ and $\omega$
is a primitive $d$-th root of unity of $\F_q$.
Several equivalent criteria for $f$ permuting $\F_{q}$ were given in~\cite{Wang-cyc},
which stated that $f$ is a PP of $\F_{q}$ if and only if
$\gcd(\prod_{i=0}^{d-1}r_{i}, s) = 1$  and
$\{a_{i}^{s} \omega^{i r_{i}}: i =0,1,\ldots, d-1\} = U_d$.
The inverses of $f$ on $\F_q$ was given in~\cite{Zheng-2,Wang-cyc2} by
\[
f^{-1}(x)
= \frac{1}{d}\sum_{i=0}^{d-1}\sum_{j=0}^{d-1}
 \omega^{i(t_{i}-jr_i)}(x/a_i)^{\widetilde{r_i} + js},
\]
where $1\le \widetilde{r_i} <s$ and $r_i\widetilde{r_i}+st_i=1$.
In~\cite{Wang-cyc2}, all involutions of the form~\eqref{cmpp} were characterized,
and a fast algorithm was provided to generate many classes of these PPs, their inverses, and involutions. The class of PPs of the form $x^rh(x^s)$ is in fact a special case of generalized cyclotomic mapping PPs.

{\bf{More general piecewise PPs}}.
The idea of more general piecewise constructions of permutations was summarized in~\cite{FH-pw,CHZ14}.
Piecewise constructions of inverses of piecewise PPs were studied in~\cite{Zheng-1,Zheng-2}.
As applications, the inverse of PP $f(x) = ax +x^{(q+1)/2}$ of $\F_q$ was given in~\cite{Zheng-1} by
\[
f^{-1}(x)=(a^2-1)^{-1}(ax -bx^{(q+1)/2}),
\]
where $(a^2-1)^{(q-1)/2}=1$, $b = (a+1)^{(q-1)/2} \in \{-1, 1\}$, and $q$ is odd.
The inverse of PP of $\F_{p^n}$ of the form
\[
(ax^{p^k} -bx +c)^{\frac{p^n +1}{2}} \pm (ax^{p^k} +bx)
\]
was obtained in~\cite{Zheng-1}, where $p$ is odd
and $a, b, c \in \F_{p^n}$.
Three classes of involutions of finite fields were also given in~\cite{Zheng-1,Zheng-2}.
In addition, the PP $f$ in~\eqref{cmpp} can be written as piecewise form,
and its inverse was deduced by the piecewise method in~\cite{Zheng-2}.

\section{The inverses of PPs of small degree}
Assume $g \in \F_q[x]$ and $b,c,d \in \F_q$ with $b \ne 0$.
Then $g$ is a PP of $\F_{q}$ if and only if $f(x)=b g(x+c) +d$ is.
By choosing $b, c, d$ suitably, we can obtain $f$ in {\bf normalized form},
that is, $f$ is monic, $f(0)=0$, and when the degree $m$ of $f$
is not divisible by the characteristic of $\F_q$,
the coefficient of $x^{m-1}$ is $0$. It suffices, therefore, to study normalized PPs.
In 1897, Dickson~\cite{Dickson1897} listed all normalized PPs of degree $\leq 5$ of $\F_{q}$ for all $q$,
and classified all PPs of degree $6$ of $\F_{q}$ for odd $q$. In 2010, the complete classification of PPs of degree $6$ and $7$ of $\F_{2^n}$ was settled in~\cite{JLi10}. For a verification of the classification of normalized PPs of degree $6$ of $\F_{q}$ for all $q$, see~\cite{PPs6}.

According to the complete classifications of PPs in \cite{Dickson1897, JLi10, PPs6},
all PPs of degree~$6$ of $\F_{q}$ for all $q$ are over small fields $\F_{q}$ with $q \leq 32$, except for $x^6$ over $\F_{2^n}$.
All PPs of degree~$7$ of $\F_{2^n}$ are over $\F_{2^n}$ with $n \leq 4$, except for $x^7$ and $x^7 +x^5 +x$. The inverses of PPs of $\F_{q}$ with $q \leq 32$ can be calculated by the Lagrange interpolation formula or Theorem~\ref{coe-f-1} in the next section. The inverses of PPs $x^6$ and $x^7$ of $\F_{2^n}$ can be obtained by the following Theorem~\ref{monomials}. The polynomial $x^7 +x^5 +x$ is actually the degree~$7$ Dickson polynomial $D_{7}(x,1)$ over $\F_{2^n}$,
and its inverse is $D_{m}(x,1)$ (by Lemma~\ref{dickson-1}), where $m$ is the inverse of  $7$ modulo $2^{2n} -1$. In other words, we obtain the inverses of all PPs of degree $6$ of $\F_{q}$ for all $q$ and the inverses of all PPs of degree $7$ of $\F_{2^n}$.

In the rest of this section, we will give the inverses of all normalized PPs of degree $\leq 5$ in~\cite{Dickson1897},
which are actually the same as that in \cite[Table 7.1]{FF} or in the previous Table~\ref{T1}.
Since the inverses of normalized PPs of small fields $\F_{q}$ with $q \leq 13$
can be obtained by the Lagrange interpolation formula,
we need only consider the normalized PPs of degree $\leq 5$ of~$\F_{q}$
for infinite many $q$.

\subsection{Inverses of monomials}
The inverse of $x$ is clearly itself, and the inverse of $x^2$ on $\F_{2^n}$ is~$x^{2^{n-1}}$.
The following theorem gives the explicit inverse of~$x^m$ on $\F_q$ for $m \geq 3$.

\begin{thm}\label{monomials}
For $m \geq 3$, if $x^m$ is a PP of $\F_{q}$,
then its inverse on $\F_{q}$ is $x^{(aq-a+1)/m}$,
where 
$a \equiv -(q-1)^{\phi(m)-1} \pmod{m}$
and $\phi$ is Euler's phi function.
\end{thm}
\begin{proof}
If $x^m$ is a PP of $\F_{q}$, then $\gcd(m, q-1) =1$ and so
$aq-a+1 = 1+a(q-1) \equiv 1 -(q-1)^{\phi(m)} \equiv 0 \pmod{m}$.
Also note that $m(aq-a+1)/m -a(q-1) =1$. Hence the inverse of $m$ modulo $q-1$ is $(aq-a+1)/m$.
\end{proof}
The proof above converts the problem of determining the inverse of $m$ modulo $q-1$ to that of computing the inverse of $q-1$ modulo $m$, and the latter is easy for small $m$.
For instance, if $m=3$ and $\gcd(3, q-1) =1$, then the inverse of~$q-1$ modulo $3$ is
$(q-1)^{\phi(3)-1} = q-1$,
and so the inverse of~$3$ modulo $q-1$ is $(aq-a+1)/3$, where $a \equiv 1-q \pmod{3}$.

\subsection{Inverses of linearized binomials and trinomials}

Assume $L_{st}(x) := bx^{q^s} +cx^{q^t}$ is an arbitrary linearized binomial of $\F_{q^n}$, where $b$, $c \in \F_{q^n}^{*}$ and $0 \leq t < s \leq n-1$. Then $L_{st}(x)= b(x^{q^{s-t}} +b^{-1}c x )\circ x^{q^t}$,
and so $L_{st}$ permutes $\F_{q^n}$ if and only if
$L_{r}(x) := x^{q^r} -ax$ permutes $\F_{q^n}$, where $r = s-t$ and $a =-b^{-1}c$.
The inverse of $L_{r}$ on $\F_{q^n}$ was given in~\cite{CoulterH04, Wu-L-bi} as follows.

\begin{thm}[\!\!\cite{CoulterH04, Wu-L-bi}]\label{L-bi}

Let $L_{r}(x) =x^{q^r} -ax$, where $a \in \F_{q^n}^{*}$ and $1 \leq r \leq n-1$.
Then $L_{r}$ is a PP of $\F_{q^n}$ if and only if the norm $\N_{q^n/q^d}(a) \neq 1$,
where $d =\gcd(n, r)$. In this case, its inverse on $\F_{q^n}$ is
\vspace{-4pt}
\[
L_{r}^{-1}(x)=\frac{\N_{q^n/q^d}(a)}{1-\N_{q^n/q^d}(a)}
\sum_{i=0}^{n/d -1}a^{-\frac{q^{(i+1)r} -1}{q^r-1}}x^{q^{i r}}.
\]
\end{thm}
The norm $\N_{q^n/q^d}(a) \neq 1$ if and only if $a$ is not a $(q^d-1)$th power.
Hence, Theorem~\ref{L-bi} gives the inverse of $x^{q^r} -ax$ for $q^r =3$, $4$, $5$ in Table~\ref{T1}.


The normalized PP of the form $x^4 +bx^2 +ax$ of $\F_{2^n}$
is the only linearized trinomial in Table~\ref{T1}.
Its inverse has a close relation with the sequence
\begin{equation}\label{ci-1}
 S_{-1} =0, ~  S_{0} =1, ~
 S_{i} = b^{2^{i-1}}S_{i-1} +a^{2^{i-1}}S_{i-2},
\end{equation}
where $1 \leq i \leq n$ and $a, b \in \F_{2^n}^{*}$.
An argument similar to that in \cite[Lemma 2]{HK10}
leads to an equivalent definition of~$S_{i}$:
\begin{equation}\label{ci-2}
S_{i} = b S_{i-1}^2 +a^2 S_{i-2}^{4}, \quad  1 \leq i \leq n.
\end{equation}
Denote $Z_{n} = S_{n} + a S_{n-2}^2$. Then
\[\begin{split}
Z_{n}^2
& = S_{n}^2  + a^2S_{n-2}^4
  \overset{\eqref{ci-1}}{=}
  b S_{n-1}^2  +a S_{n-2}^2  + a^2S_{n-2}^4  \\
& \overset{\eqref{ci-2}}{=}
  S_{n}  +a S_{n-2}^2
 = Z_{n},
\end{split}\]
and so $Z_{n}=0$ or $1$. A criterion for $f(x)= x^{q^2} +bx^q +ax$ to be a PP of $\F_{q^n}$
and the inverse of $f$ on $\F_{q^n}$ were presented in \cite[Theorem 3.2.29]{Wu-PhD}.
Taking $q=2$ in this theorem and using the fact $Z_{n}=0$ or $1$,
we obtain the following result.

\begin{cor}\label{L-tri}
Let $L(x) = x^4 +bx^2 +ax$, where $a, b \in \F_{2^n}^{*}$ and $n \geq 1$.
Then $L$ is a PP of $\F_{2^n}$ if and only if
$S_{n} + a S_{n-2}^2=1$.
In this case, the inverse of $L$ on $\F_{2^n}$ is
\[
L^{-1}(x) = \sum_{i=0}^{n-1}\big( S_{n-2-i}^{2^{i+1}} +a^{1 -2^{i+1}} S_{i} \big) x^{2^i}.
\]
\end{cor}

Note that Corollary~\ref{L-tri} holds for $n =1$, $2$.
Indeed, if $n=1$ then $L(x) \equiv L^{-1}(x) \equiv x \pmod{x^2 +x}$.
If $n=2$ and $L$ is a PP of $\F_{4}$,
then $L(x) \equiv L^{-1}(x) \equiv bx^2 \pmod{x^4 +x}$.

The necessary and sufficient condition for $L$ permuting $\F_{2^n}$
can also be obtained by \cite[Proposition 2]{HK10}.
This proposition also shown that $M_0 = (2^n - (-1)^n)/3$,
where $M_0$ is the number of $c \in \F_{2^n}^{*}$
such that $P_c(x) = x^3 +x +c$ has no root in $\F_{2^n}$.
Since $L(b^{1/2}x)= b^{2}x(x^3 +x + ab^{-3/2})$, $L$ is a PP of $\F_{2^n}$
if and only if $P_c$ has no root in $\F_{2^n}$, where $c =ab^{-3/2}$.
Hence the number of $a, b \in \F_{2^n}^{*}$ such that $L$ permutes $\F_{2^n}$
is equal to $(2^n -1)(2^n - (-1)^n)/3$, which implies the probability of
$L$ permuting $\F_{2^n}$ is almost $1/3$.

In Corollay~\ref{L-tri}, let $a=b=1$. Then
$S_{i}=S_{i-1}+S_{i-2}$ by~\eqref{ci-1}, and so
$(S_{-1}, S_{0}, S_{1}, S_{2}, S_{3}, S_{4},\ldots) = (0,1,1,0,1,1,\ldots)$.
Thus we obtain the following result.
\begin{cor}
Let $L(x) = x^4 +x^2 +x$. Then $L$ is a PP of $\F_{2^n}$ if and only if
$n \equiv 1, 2 \pmod{3}$. In this case, the inverse of $L$ on $\F_{2^n}$ is
$L^{-1}(x) = \sum_{i=0}^{n-1} x^{2^i}$ with $i \not\equiv 2-n \pmod{3}$.
\end{cor}

\vspace{-6pt}
\subsection{Inverses of non-linearized trinomials}

In Table~\ref{T1}, there are only two infinite classes of
non-linearized permutation trinomials. One is the polynomial
$x^5 + ax^3 +5^{-1}a^2x$, where $a \in \F_{q}$ and $q \equiv \pm 2 \pmod{5}$.
It is actually the Dickson PP $D_5(x, -5^{-1}a)$,
and by Lemma~\ref{dickson-1} its inverse on $\F_{q}$ is $D_m(x, -(5^{-1}a)^5)$, where $m =(3q^2-2)/5$ (by the proof of Theorem~\ref{monomials}). The other is as follows.

\begin{lem}[\!\!{\cite[Table~7.1]{FF}}]\label{lem-f=x5x3x}
Let $f(x)=x^5 -2ax^3 +a^2x$, where $a \in \F_{5^n}^*$ and $n \geq 1$.
Then $f$ is a PP of $\F_{5^n}$ if and only if $a^{(q-1)/2} =-1$.
\end{lem}

The inverse of $f$ was given in~\cite{KLi18-1} by solving equations over finite fields.

\begin{thm}[\!\!{\cite[Lemma 4.9]{KLi18-1}}]\label{x5-1-KLi}
The inverse of $f$ in Lemma~\ref{lem-f=x5x3x} on $\F_{5^n}$ is
\[
f^{-1}(x)= x \bigg(\frac{(a/x^{2})^{\frac{5^n-1}{4}}}{1-(a/x^{2})^{\frac{5^n-1}{4}}}
\sum_{i=0}^{n-1}(a/x^{2})^{-\frac{5^{i+1} -1}{4}}(1/x^{2})^{5^{i}}\bigg)^2,
\]
where $x \in \F_{5^n}^{*}$ and $f^{-1}(0) =0$.
\end{thm}

By employing the method in the next section,
we obtain the explicit polynomial form of $f^{-1}$ as follows.

\begin{thm}\label{x5-1}
The inverse of $f$ in Lemma~\ref{lem-f=x5x3x} on $\F_{5^n}$ is
\[
f^{-1}(x)=\sum_{0 \leq i \leq j \leq n-1}
\raisebox{-2pt}{$
a^{-\frac{5^n +5^{i+1} +5^{j+1} -3}{4}} b_{ij} x^{\frac{5^n +5^i +5^j -1}{2}}
$,}
\]
where $b_{ij} = 1$ if $i<j$ and $b_{ij} = 3$ if $i=j$.
\end{thm}

\begin{rem}
Theorem~\ref{x5-1} can be obtained from Theorem~\ref{x5-1-KLi} and
\[
\Big(\sum_{1 \leq i \leq n} x_i \Big)^2
= \sum_{1 \leq i \leq n} x_i^2 +\sum_{1 \leq i < j \leq n} 2 x_i x_j,
\]
where $x_i = a^{-\frac{5^{i}-1}{4}}x^{\frac{5^{i-1}-1}{2}}$.
However, we will demonstrate our method of deducing Theorem~\ref{x5-1} in the next sections.
The main reason is that our method can also be used to find the inverses of other PPs of small degree; see for example~\cite{ZhengWLW19}.
\end{rem}

In summary, all inverses of normalized PPs of degree $\leq 5$ are obtained.
We list these PPs and their inverses in Table~\ref{T1}.

\section{The coefficients of inverse of a PP}

In this section, we will write the coefficients of inverse of the PP $f$
in Lemma~\ref{lem-f=x5x3x} in terms of binomial coefficients,
by employing the following formula~\eqref{f-1} presented first in~\cite{MR-1}.

\begin{thm}[See \cite{MR-1}]\label{coe-f-1}
  Let $f \in \F_q[x]$ be a PP of\, $\F_q \color{blue}{(q \ge 3)}$
  such that $f(0)=0$, and let
\[
   f(x)^{q-1-i} \equiv \sum_{\color{blue}{ 1 \leq k \leq q-1}} b_{ik}x^k\pmod{x^q-x},
\]
where $i = 1, 2, \ldots, q-2$. Then the inverse of $f$ on $\F_q$ is
\begin{equation}\label{f-1}
f^{-1}(x)=\sum_{1 \leq i \leq q-2} b_{i,q-2} x^i.
\end{equation}
\end{thm}
\begin{proof}
Assume $f^{-1}(x)=\sum_{i=1}^{q-2} c_i x^i$. From the Lagrange interpolation formula, we have
\begin{align*}
f^{-1}(x)
& =\sum_{a \in \F_q} a \big(1-(x-f(a))^{q-1}\big) \\
& = \sum_{a \in \F_q^*} a \Big(-\sum_{1 \leq i \leq q-1}(-1)^i(-f(a))^{q-1-i}x^i\Big) \\
& = \sum_{1 \leq i \leq q-1}\Big(-\sum_{a \in \F_q^*} a f(a)^{q-1-i}\Big)x^i.
\end{align*}
Hence for $1 \leq i \leq q-2$, we have
\begin{align*}
c_i & = -\sum_{a \in \F_q^*} a f(a)^{q-1-i}
      = -\sum_{a \in \F_q^*} a\sum_{1 \leq k \leq q-1} b_{ik}a^k  \\
    & = -\sum_{1 \leq k \leq q-1} b_{ik}\sum_{a \in \F_q^*} a^{k+1}
      = b_{i,q-2},
\end{align*}
where the last identity follows from

\[
\sum_{a \in \F_{q}}a^t
 = \begin{cases}
  -1 & \text{if $t = q-1$,}\\
  ~~0  & \text{if $t = 0$, $1$, $\ldots$, $q-2$.}
  \end{cases}
\]
The proof is completed.
\end{proof}
\begin{rem}
Theorem~\ref{coe-f-1} is the same as the one in~\cite{MR-1, Wang-1}.
All these results are essentially part of Theorem~2 in~\cite{MMW16}.
For the reason of completeness, we include a proof by using the Lagrange
interpolation formula.
\end{rem}

Next we use Theorem~\ref{coe-f-1} to calculate the coefficients of the inverse of $f$ in Lemma~\ref{lem-f=x5x3x}.
Recall that $f(x)=x^5 -2ax^3 +a^2x$, where $a \in \F_{5^n}$ and $a^{(q-1)/2}=-1$.
Let $q=5^n$ and $r_i=q-1-i$, where $1 \leq i \leq q-2$. Then
\begin{equation}\label{x5x3x^ri}
\begin{split}
  f(x)^{r_i} & = x^{r_i}(x^2 -a)^{2r_i}  \\
  & = \sum_{0 \leq j \leq 2r_i}\binom{2r_i}{j}(-a)^{2r_i-j}x^{r_i+2j}.
\end{split}
\end{equation}
The degree of $f(x)^{r_i}$ is $5r_i$, and $5 \leq 5r_i < 4(q-1)+q-2$.
By~\eqref{f-1}, the coefficient $b_{i,q-2}$ of $f^{-1}$ equals the sum
of the coefficients of $x^{k(q-1)+(q-2)}$ ($k=0,1,2,3$) in~\eqref{x5x3x^ri}.
If $i$ is even, then $r_i+2j$ is even, and so
the coefficients of odd powers of $x$ in~\eqref{x5x3x^ri} are all $0$.
Also note $k(q-1)+(q-2)$ is odd. We have $b_{i,q-2} =0$ for even $i$. If $i$ is odd, then
\[
b_{i,q-2}=\sum_{0 \leq k \leq 3}
\binom{2r_i}{(k(q-1)+i-1)/2}(-a)^{2r_i-\frac{k(q-1)+i-1}{2}}.
\]
Let $i=2m+1$. Then $0 \leq m \leq (q-3)/2$ and
\begin{equation}\label{bi-sum}
\begin{split}
b_{i,q-2}
& = \sum_{0 \leq k \leq 3}
\binom{2q-4m-4}{k\frac{q-1}{2} +m}(-a)^{-k\frac{q-1}{2}-5m-2} \\
& = \sum_{0 \leq k \leq 3}
\binom{2q-4m-4}{k\frac{q-1}{2} +m}(-1)^{k+m}a^{-5m-2},
\end{split}
\end{equation}
where the last identity follows from the fact $a^{(q-1)/2}=-1$ and $q=5^n$.
If $2q-4m-4 < k(q-1)/2  +m$, i.e.,
\[
m > ((4-k)q +k-8)/10,
\]
then $\binom{2q-4m-4}{k\frac{q-1}{2} +m} =0$.
Thus a direct computation reduces~\eqref{bi-sum}~to
\begin{equation}\label{bi=e-mk}
b_{i,q-2} =
\begin{cases}
0                            & \makebox[60pt][r]{if~~$4T+1 < m \leq (q-3)/2$,}   \\
e_{m0}                       & \makebox[60pt][r]{if~~$3T < m \leq 4T+1$,}  \\
e_{m0}+e_{m1}                & \makebox[60pt][r]{if~~$2T < m \leq 3T$,}  \\
e_{m0}+e_{m1}+e_{m2}         & \makebox[60pt][r]{if~~$T < m \leq 2T$,}   \\
e_{m0}+e_{m1}+e_{m2}+e_{m3}  & \makebox[60pt][r]{if~~$0\leq m \leq  T$,} \\
\end{cases}
\end{equation}
where $T = (q-5)/10 = (5^{n-1} -1)/2$ and
\begin{equation}\label{e-mk}
e_{mk} = \binom{2q-4m-4}{k\frac{q-1}{2} +m}(-1)^{k+m}a^{-5m-2}, k=0,1,2,3.
\end{equation}
Now the key of deducing Theorem~\ref{x5-1} is to find the values of binomial coefficients above.
\section{Explicit values of binomial coefficients}

In this section, we first give the explicit values of binomial coefficients in~\eqref{e-mk},
and then prove Theorem~\ref{x5-1}.
In order to remove the multiples of $q$ in these binomial coefficients, we need two lemmas.

\begin{lem}[Lucas' theorem] For non-negative integers $n$, $k$ and a prime $p$, let
$n = n_{0}+n_{1}p+\cdots +n_{s}p^{s}$ and $k  = k_{0}+k_{1}p+\cdots +k_{s}p^{s}$
be their $p$-adic expansions, where $0 \le n_i$, $k_i \leq p-1$ for $i=0,1,\ldots,s$. Then
\[
 {\binom {n}{k}}\equiv \prod_{i=0}^{s}{\binom {n_{i}}{k_{i}}}\pmod{p}
\]
(with the convention $\binom {0}{0} =1$ and $\binom {n}{k} = 0$ if $n < k$).
In particular, $\binom {n}{k} \!\not\equiv 0\pmod{p}$
if and only if $n_i \geq k_i$~for~all~$i$.
\end{lem}
\begin{lem}\label{bico-qmodp}
Let $q$ be a power of a prime $p$, and let $r$, $k$ be integers with $0 \leq k \leq q-1$. Then
\[
\binom {q+r}{k} \equiv \binom {r}{k} \pmod{p}.
\]
where $\binom {n}{k}=n(n-1)\cdots(n-k+1)/k!$.
\end{lem}
\begin{proof}
  By the Chu-Vandermonde identity, we have
  \begin{equation*}
  \binom{q+r}{k} = \sum_{i=0}^{k}\binom {q}{i}\binom {r}{k-i}
  \equiv \binom {r}{k} \pmod{p},
  \end{equation*}
where we use the fact $\binom{q}{i} \equiv 0 \pmod{p}$ for $1 \leq i \leq q-1$.
\end{proof}

In~\eqref{bi=e-mk} we defined $T=(q-5)/10$ with $q=5^{n}$.
Then for $0 \leq m \leq 4T+1$, applying Lemma~\ref{bico-qmodp} twice yields that
\begin{equation}\label{e-m0}
\begin{split}
\binom{2q-4m-4}{m}
& \equiv \binom{-4m-4}{m} \\
& \equiv (-1)^m \binom{5m+3}{m}  \pmod{5},
\end{split}
\end{equation}
where we use the fact
$\binom {-n}{m} = (-1)^m \binom {m+n-1}{m}$ for $m$, $n > 0$.
Similarly, for $0 \leq m \leq 3T$,
\begin{equation}\label{e-m1}
\binom{2q-4m-4}{\frac{q-1}{2} +m} \equiv
(-1)^{m} \binom{5m+3 +\frac{q-1}{2}}{m +\frac{q-1}{2}} \pmod{5},
\end{equation}
For $0 \leq m \leq 2T$, we have $q-4m-4 > 0$ and, by Lucas' theorem and Lemma~\ref{bico-qmodp},
\begin{equation}\label{e-m2}
\begin{split}
\binom{2q-4m-4}{q+m-1}
& \equiv \binom{q-4m-4}{m-1} \equiv \binom{-4m-4}{m-1}\\
& \equiv (-1)^{m-1} \binom{5m+2}{m -1} \pmod{5}.
\end{split}
\end{equation}
Similarly, for $0 \leq m \leq T$,
\begin{equation}\label{e-m3}
\binom{2q-4m-4}{3\frac{q-1}{2} +m} \equiv
(-1)^{m-1} \binom{5m+2 +\frac{q-1}{2}}{m -1 +\frac{q-1}{2}} \pmod{5}.
\end{equation}

Next we use Lucas' theorem to find the explicit value
of the last binomial coefficients in~\eqref{e-m0}-\eqref{e-m3}.

\begin{thm}\label{Aneq0}
Let $0 \leq m \leq 5^{n}-1$ with $n \geq 1$.  Write
\[
m = m_0 +m_1 5  +\cdots +m_{n-1} 5^{n-1}, ~~0 \le m_i \le 4.
\]
Then the following three statements are equivalent:
\begin{enumerate}[$(i)$]
  \item $\binom{5m+3}{m} \not\equiv 0 \pmod{5}$;
  \item $3 \geq m_0 \geq m_1 \geq \cdots \geq m_{n-1} \geq 0$;
  \item $m = \frac{5^{k_1} -1}{4} + \frac{5^{k_2} -1}{4} + \frac{5^{k_3} -1}{4}$,
     where $0 \leq k_1 \leq k_2 \leq k_3 \leq n$.
\end{enumerate}
\end{thm}
\begin{proof}
It is easy to obtain the $5$-adic expansions:
 \begin{equation}\label{5m+3}
 \begin{split}
   5m+3 &=   3  +m_0 5  +\cdots+m_{n-2}5^{n-1} +m_{n-1}5^{n}, \\
   m &= m_0  +m_1 5  +\cdots+m_{n-1}5^{n-1}.
 \end{split}
 \end{equation}
By Lucas' theorem, (i) is equivalent to (ii).
To show (ii) is equivalent to (iii), it suffices to prove $M_1 = M_2$, where
\[\begin{split}
  M_1 & = \{m: ~ 3 \geq m_0 \geq m_1 \geq \cdots \geq m_{n-1} \geq 0\}, \\
  M_2 & = \{\tfrac{5^{k_1} -1}{4} + \tfrac{5^{k_2} -1}{4} + \tfrac{5^{k_3} -1}{4}:
        ~ 0 \leq k_1 \leq k_2 \leq k_3 \leq n \}.
\end{split}\]
Since $(5^{k} -1)/4 = 1 +1\cdot5  +\cdots +1\cdot5^{k-1}$
and $0 \leq k_1 \leq k_2 \leq k_3 \leq n$, we have $M_2 \subseteq M_1$.
It remains to show that $M_1 \subseteq M_2$, i.e., $m \in M_2$ for any $m \in M_1$.
The remainder of our proof is divided into two cases.

Case 1: assume $m \in M_1$ such that $m_0 = \cdots = m_{n-1} = a$,
where $0 \leq a \leq 3$. If $a=2$, then
$m =\tfrac{5^{k_1} -1}{4} + \tfrac{5^{k_2} -1}{4} + \tfrac{5^{k_3} -1}{4}$,
where $k_1 = 0$ and $k_2 = k_3 =n$.
Hence $m \in M_2$. Similarly, $m \in M_2$ for $a=0$, $1$ or $3$.

Case 2: assume $m \in M_1$ such that $m_0$, $m_1$, $\ldots$, $m_{n-1}$ are not all equal.
Then the number $N$ of the sign $>$ in the inequality
$
3 \geq m_0 \geq m_1 \geq \cdots \geq m_{n-1} \geq 0
$
is 1, 2 or 3. If $N=3$, then there exist $a$, $b$, $c$ such that
\[\begin{split}
0 &\leq a < b < c \leq n-2, \\
m_{0}   & = m_{ 1}  = \cdots = m_{a} = 3,   \\
m_{a+1} & = m_{a+2} = \cdots = m_{b} = 2,   \\
m_{b+1} & = m_{b+2} = \cdots = m_{c} = 1,   \\
m_{c+1} & = m_{c+2} = \cdots = m_{n-1} = 0. \\
\end{split}\]
Then $m =\tfrac{5^{k_1} -1}{4} + \tfrac{5^{k_2} -1}{4} + \tfrac{5^{k_3} -1}{4}$,
where $k_1 = a+1$, $k_2 = b+1$ and $k_3 = c+1$.
Hence $m \in M_2$. Similarly, $m \in M_2$ for $N= 1$ or $2$.
\end{proof}

Two criteria that $\binom{5m+3}{m} \not\equiv 0 \pmod{5}$ are given in the theorem above.
The following theorem finds the explicit values of this class of binomial coefficients.

\begin{thm}\label{Aeq13}
Let $m = \frac{5^{k_1} -1}{4} + \frac{5^{k_2} -1}{4} + \frac{5^{k_3} -1}{4}$
with $0 \leq k_1 \leq k_2 \leq k_3 \leq n$. Then in $\F_5$,
\[
\binom{5m+3}{m}=
\begin{cases}
  1 & \text{if $k_1 = k_2 = k_3$ or $k_1 < k_2 < k_3$,}\\
  3 & \text{if $k_1 = k_2 < k_3$ or $k_1 < k_2 = k_3$.}
\end{cases}
\]
\end{thm}
\begin{proof}
For ease of notations, we denote $A = \binom{5m+3}{m}$.

Case 1: $k_1 = k_2 = k_3 =k$ with $0 \leq k \leq n$. If $k =0$ then $m = 0$
and $A = \binom{3}{0} = 1$. If $1 \leq k \leq n$, then $m = 3(5^k -1)/4$, i.e.,
$m_0 =\cdots =m_{k-1} =3$ and $m_k =\cdots =m_{n-1} =0$.
By Lucas' theorem and~\eqref{5m+3}, $A \equiv \binom{3}{3}\binom{3}{0} = 1 \pmod{5}$.

Case 2: $k_1 = k_2 < k_3$. If $k_1 = k_2 =0$, then
\[
  m_0     =\cdots =m_{k_3 -1} =1,
~~m_{k_3} =\cdots =m_{n   -1}    =0.
\]
Thus $A \equiv \binom{3}{1}\binom{1}{0} = 3 \pmod{5}$.
If $k_1 = k_2 =k \geq 1$, then
$m_0     =\cdots =m_{k   -1} =3$,
$ m_k     =\cdots =m_{k_3 -1} =1$,
$ m_{k_3} =\cdots =m_{n   -1} =0$.
Hence $A \equiv \binom{3}{3}\binom{3}{1}\binom{1}{0} = 3 \pmod{5}$.

Case 3: $k_1 < k_2 = k_3$.
The proof is similar to that of Case~2 and so is omitted.

Case 4: $k_1 < k_2 < k_3$. If $k_1 =0$ then
$m_0     =\cdots =m_{k_2 -1} =2$,
$m_{k_2} =\cdots =m_{k_3 -1} =1$,
$m_{k_3} =\cdots =m_{n   -1} =0$.
Hence $A \equiv \binom{3}{2}\binom{2}{1}\binom{1}{0} \equiv 1 \pmod{5}$.
If $0< k_1 < k_2 < k_3$, then
\[\begin{split}
m_0   &  =\cdots =m_{k_1 -1} =3,
~~m_{k_1} =\cdots =m_{k_2 -1} =2, \\
m_{k_2}& =\cdots =m_{k_3 -1} =1,
~~m_{k_3} =\cdots =m_{n   -1} =0.
\end{split}\]
Hence $A \equiv \binom{3}{3}\binom{3}{2}\binom{2}{1}\binom{1}{0} \equiv 1 \pmod{5}$.
\end{proof}

\begin{cor}\label{em0=13}
Let $T =(5^{n-1}-1)/2$, $n \geq 1$ and $2T < m \leq 4T +1$. Then in $\F_5$,
\[
\binom{5m+3}{m} =
 \begin{cases}
  3^{\binom{k_1}{k_2}}  & \text{if $m = (5^{n} +5^{k_1} +5^{k_2} -3)/4$,} \\
  ~~0  & \text{otherwise,}
 \end{cases}
\]
where $0 \leq k_1 \leq k_2 \leq n-1$.
\end{cor}
\begin{proof}
Let $A = \binom{5m+3}{m}$. According to Theorem~\ref{Aneq0}, if $0\leq m \leq 5^{n}-1$
then $A \not\equiv 0 \pmod{5}$ if and only if
\[
m = \tfrac{5^{k_1} -1}{4} + \tfrac{5^{k_2} -1}{4} + \tfrac{5^{k_3} -1}{4},
\]
where $0 \leq k_1 \leq k_2 \leq k_3 \leq n$. Since $2T+1 = 5^{n-1}$ and
\[
4T +1 = 4 +4\cdot5  +\cdots  +4\cdot5^{n-2} +1\cdot5^{n-1},
\]
we have, for $2T < m \leq 4T +1$, $A \not\equiv 0 \pmod{5}$ if and only if
$m = (5^{k_1} +5^{k_2} +5^{n}-3)/4$, where $0 \leq k_1 \leq k_2 < n$.
In this case, $A \equiv  3^{\binom{k_1}{k_2}}\pmod{5}$ by Theorem~\ref{Aeq13}.
\end{proof}

The following corollary presents a congruence relation for the binomial coefficients in~\eqref{e-m0} and \eqref{e-m2}.

\begin{cor}\label{em02=0}
Let $T =(5^{n-1}-1)/2$, $n \geq 2$ and $T < m \leq 2T$. Then
\[
\binom{5m+3}{m}\equiv \binom{5m+2}{m -1} \pmod{5}.
\]
\end{cor}
\begin{proof}
Denote by $A$ and $B$ the above binomial coefficients, respectively.
Then $mA =(5m+3)B$, and so $2mA \equiv B \pmod{5}$.
If $A \equiv 0 \pmod{5}$, then $B \equiv 0 \pmod{5}$, and thus $A \equiv B \pmod{5}$.
We next show $A \equiv B \pmod{5}$ for $A \not\equiv 0 \pmod{5}$.
By Theorem~\ref{Aneq0}, for $0\leq m \leq 5^{n}-1$, $A \not\equiv 0 \pmod{5}$ if and only if
$
m = \tfrac{5^{k_1} -1}{4} + \tfrac{5^{k_2} -1}{4} + \tfrac{5^{k_3} -1}{4}
$,
where $0 \leq k_1 \leq k_2 \leq k_3 \leq n$. Since
\[\begin{split}
 T  & = 2 +2\cdot5  +\cdots  +2\cdot5^{n-2}, \\
2T  & = 4 +4\cdot5  +\cdots  +4\cdot5^{n-2},
\end{split}\]
we obtain, for $T <  m \leq 2T$, $A \not\equiv 0 \pmod{5}$ if and only if
$m = \frac{5^{k_{1}} -1}{4} +T$, where $1 \leq k_{1} \leq n-1$.
Hence $m \equiv 3 \pmod{5}$, and so $A \equiv 2mA \equiv B \pmod{5}$.
\end{proof}

Next we study the last binomial coefficient in~\eqref{e-m1}.
\begin{thm}\label{bico-id-1}
Let $0 \leq m \leq 5^{n}-1$ with $n \geq 1$.  Write
\[
m = m_0 +m_1 5 +\cdots +m_{n-1} 5^{n-1}, ~~0 \le m_i \le 4.
\]
Then the following three statements are equivalent:
\begin{enumerate}[$(i)$]
  \item {\Large$\binom{5m+3 +\frac{5^{n}-1}{2}}{m +\frac{5^{n}-1}{2}}$} $\not\equiv 0 \pmod{5}$;
  \item $3 = m_0 \geq m_1 \geq \cdots \geq m_{n-1} \geq 2$;
  \item $m = \frac{5^{n}-1}{2} + \frac{5^{k} -1}{4}$, where $1 \leq k \leq n$.
\end{enumerate}
\end{thm}
\begin{proof}
Denote $\alpha = 5m+3 +\tfrac{5^{n}-1}{2}$ and $\beta = m +\tfrac{5^{n}-1}{2}$.
Then their $5$-adic expansions are as follows:
\begin{equation}\label{alpha}
\begin{split}
\alpha = & ~0+(m_0+3)5 +(m_1+2)5^2 +(m_2+2)5^3+\cdots\\
&+(m_{n-2}+2)5^{n-1} +m_{n-1}5^{n},
\end{split}
\end{equation}
\begin{equation}\label{beta}
\begin{split}
\beta = & ~(m_0+2)+(m_1+2)5 +(m_2+2)5^2 +\cdots \\
& +(m_{n-1}+2)5^{n-1} +0\cdot 5^{n}.
\end{split}
\end{equation}
If $\binom{\alpha}{\beta} \not\equiv 0 \pmod{5}$, then by Lucas' theorem,
$0 \succeq m_0 +2$, i.e., $m_0 =3$,
where $a_i \succeq b_i$ denotes $a_i \geq b_i$ in $\F_5$.
The condition $m_0 =3$ leads to a carry 1 in~\eqref{alpha} and~\eqref{beta}, respectively. Then
\[
\begin{split}
\alpha = & ~0+1\cdot 5 +(m_1+3)5^2 +(m_2+2)5^3 +\cdots, \\
\beta  = & ~0+(m_1+3)5 +(m_2+2)5^2 +(m_3+2)5^3 +\cdots. \\
\end{split}
\]
If $\binom{\alpha}{\beta} \not\equiv 0 \pmod{5}$, then $1 \succeq m_1+3$,
i.e., $m_{1}=2,3$, which also yields a carry 1 in the expansions above. Now
\[
\begin{split}
\alpha = & ~0+1\cdot 5 +(m_1-2)5^2 +(m_2+3)5^3 +\cdots, \\
\beta  = & ~0+(m_1-2)5 +(m_2+3)5^2 +(m_3+2)5^3 +\cdots. \\
\end{split}
\]
If $\binom{\alpha}{\beta} \not\equiv 0 \pmod{5}$, then $1 \succeq m_1-2 \succeq m_2+3$,
and so $1 \succeq m_2+3$, i.e., $m_{2}=2,3$, which also yields a carry 1.
And so on, we obtain $\binom{\alpha}{\beta} \not\equiv 0 \pmod{5}$ if and only if
\begin{equation}\label{b1-3-2}
m_0 = 3, ~1 \succeq m_1-2 \succeq \cdots \succeq m_{n-1}-2,  ~m_{n-1}+1 \succeq 1.
\end{equation}
So $m_k = 2$, $3$ for $1 \leq k \leq n-1$. There are exactly two cases:
\begin{itemize}
  \item $m_{1} = \cdots = m_{n-1}=3$, i.e., $m = \frac{5^{n}-1}{2} + \frac{5^{n} -1}{4}$;
  \item $m_{1} = \cdots = m_{k-1}=3$ and $m_{k} = \cdots = m_{n-1}=2$ for some $1 \leq k \leq n-1$.
   That is, $m = \frac{5^{n}-1}{2} + \frac{5^{k} -1}{4}$, where $1 \leq k \leq n-1$.
\end{itemize}
Therefore, \eqref{b1-3-2} is equivalent to $(ii)$ or $(iii)$.
\end{proof}

\begin{cor}\label{em1=0}
Let $0 \leq m \leq (5^{n}-1)/2$ with $n \geq 1$. Then
\[
\binom{5m+3 +\frac{5^{n}-1}{2}}{m +\frac{5^{n}-1}{2}} \equiv 0 \pmod{5}.
\]
\end{cor}

Now we consider the last binomial coefficient in~\eqref{e-m3}.
\begin{thm}\label{Cneq0}
Let $T =(5^{n-1}-1)/2$, $n \geq 1$ and $0 \leq m \leq 2T$.  Write
$m = m_0 +m_1 5  +\cdots +m_{n-2} 5^{n-2}$, where $0 \le m_i \le 4$.
Then the following three statements are equivalent:
\begin{enumerate}[$(i)$]
  \item {\Large$\binom{5m+2+\frac{5^{n}-1}{2}}{m-1 +\frac{5^{n}-1}{2}}$} $\not\equiv 0 \pmod{5}$;
  \item $2 \geq m_0 \geq m_1 \geq \cdots \geq m_{n-2} \geq 0$;
  \item $m = \frac{5^{k_1} -1}{4} + \frac{5^{k_2} -1}{4}$,
        where $0 \leq k_1 \leq k_2 \leq n-1$.
\end{enumerate}
\end{thm}
\begin{proof}
Denote $\alpha = 5m+2 +\tfrac{5^{n}-1}{2}$ and $\beta = m -1+\tfrac{5^{n}-1}{2}$.
Then their $5$-adic expansions are as follows:
\[\begin{split}
& \alpha = 4 +(m_0+2)5 +\cdots +(m_{n-2}+2)5^{n-1}, \\
& \beta = (m_0+1) +(m_1+2)5  +\cdots +2\cdot5^{n-1}.
\end{split}\]
The proof that $(i)$ is equivalent to $(ii)$ is divided in two cases.

Case 1: $0 \leq m_i \leq 2$ for all $i$. Then by Lucas' Theorem,
$\binom{\alpha}{\beta} \not\equiv 0 \pmod{5}$ if and only if
\[\begin{split}
4 \geq m_0&+1,\,\, m_0 +2 \geq \cdots \geq m_{n-2}+2 \geq 2, \text{~~i.e.},  \\
& 2 \geq m_0 \geq \cdots \geq m_{n-2} \geq 0.
\end{split}\]

Case 2: $m_i =3$ or~$4$ for some $i$ ($0 \leq i \leq n-2$).
We first show $\binom{\alpha}{\beta} \equiv 0 \pmod{5}$ if $m_0 =3$.
An argument similar to the one used in Theorem~\ref{bico-id-1} shows that
$\binom{\alpha}{\beta} \not\equiv 0 \pmod{5}$ if and only if
$m_1=3$ and $1 \geq m_2 -2 \geq \cdots \geq  m_{n-2} -2 \geq 3$.
This is contrary to $3 > 1$.  Similarly, we have
$\binom{\alpha}{\beta} \equiv 0 \pmod{5}$ when $m_0 =4$, $m_i =3$ or $4$ for $1 \leq i \leq n-2$.

An argument similar to the one used in Theorem~\ref{Aneq0}
can show that $(ii)$ is equivalent to~$(iii)$.
\end{proof}

The following corollaries give the linear congruence relations
between the binomial coefficients in~\eqref{e-m0}, \eqref{e-m2} and \eqref{e-m3}.

\begin{cor}\label{Ceq-D}
Let $T =(5^{n-1}-1)/2$, $n \geq 1$ and $0 \leq m \leq 2T$. Then
\[
  \binom{5m+2+\frac{5^{n}-1}{2}}{m-1 +\frac{5^{n}-1}{2}}
  \equiv -\binom{5m+2}{m} \pmod{5}.
\]
\end{cor}
\begin{proof}
Denote by $C$ and $D$ the above binomial coefficients, respectively.
Then by Lucas' theorem, $D \not\equiv 0 \pmod{5}$ if and only if
$2 \geq m_0 \geq m_1 \geq \cdots \geq m_{n-2} \geq 0$.
That is, $D \not\equiv 0 \pmod{5}$ if and only if $C \not\equiv 0 \pmod{5}$.
Hence $C \equiv -D \pmod{5}$ when $C \equiv 0 \pmod{5}$.
On the other hand, if $C \not\equiv 0 \pmod{5}$, then, by Theorem~\ref{Cneq0},
$m = \frac{5^{k_1} -1}{4} + \frac{5^{k_2} -1}{4}$,
where $0 \leq k_1 \leq k_2 \leq n-1$.
An argument similar to that in Theorem~\ref{Aeq13} shows
$C \equiv 3 +\binom{k_1}{k_2} \equiv -D \pmod{5}$.
Hence $C \equiv -D \pmod{5}$ for $C \not\equiv 0 \pmod{5}$.
\end{proof}

\begin{cor}\label{em023=0}
Let $T =(5^{n-1}-1)/2$, $n \geq 2$ and $0 \leq m \leq T$. Then
\[
\binom{5m+3}{m} + \binom{5m+2 +\frac{5^{n}-1}{2}}{m -1 +\frac{5^{n}-1}{2}}
\equiv \binom{5m+2}{m -1}  \pmod{5}.
\]
\end{cor}
\begin{proof}
Denote by $A$, $C$, $B$ the above binomial coefficients, respectively.
By $mA = (5m +3)B$, we get $2mA \equiv B \pmod{5}$. Let $D =\binom{5m+2}{m}$.
Then $(4m+3)A = (5m+3)D$, and so $D \equiv (3m+1)A \pmod{5}$. By Corollary~\ref{Ceq-D},
$C \equiv -D \equiv (2m-1)A \pmod{5}$. Hence $A + C \equiv B \pmod{5}$.
\end{proof}

With the help of the preceding results, we now prove Theorem~\ref{x5-1}.
According to~\eqref{bi=e-mk}, the coefficients $b_{i,q-2}$ of $f^{-1}$ are linear combinations
of $e_{m0}$, $e_{m1}$, $e_{m2}$, $e_{m3}$ defined by~\eqref{e-mk}.
Corollary~\ref{em1=0} and the congruence~\eqref{e-m1} imply that
\[
e_{m1} \equiv 0 \pmod{5}  \quad \text{$0 \leq m \leq  3T$,}
\]
where $T =(5^{n-1}-1)/2$. By~\eqref{e-m0}, \eqref{e-m2} and Corollary~\ref{em02=0},
\[
e_{m0} +e_{m2} \equiv 0 \pmod{5} \quad \text{for $T < m \leq 2T$.}
\]
From~\eqref{e-m0}, \eqref{e-m2}, \eqref{e-m3} and Corollary~\ref{em023=0}, we obtain
\[
e_{m0}+e_{m2}+e_{m3} \equiv 0 \pmod{5}  \quad \text{for $0\leq m \leq T$.}
\]
By~\eqref{e-m0} and Corollary~\ref{em0=13}, for $2T < m \leq 4T+1$ we get,~in~$\F_5$,
\[
e_{m0} =
 \begin{cases}
  3^{\binom{k_1}{k_2}} a^{-5m-2}
  & \text{if $m = (5^{n} +5^{k_1} +5^{k_2} -3)/4$,} \\
  ~~0  & \text{otherwise,}
 \end{cases}
\]
where $0 \leq k_1 \leq k_2 \leq n-1$.
Then Theorem~\ref{x5-1} follows from~\eqref{bi=e-mk} and Theorem~\ref{coe-f-1}
when $n \geq 2$
(since $n \geq 2$ is a necessary condition of Corollaries~\ref{em02=0} and \ref{em023=0}).
In addition, it is easy to verify that Theorem~\ref{x5-1} also holds for $n=1$.


%

\section*{Acknowledgment}

We are grateful to the referees and the editor for many useful comments and suggestions.
We would like to thank Junwei Guo, Fu Wang, Baofeng Wu for their helpful suggestions.
The work was partially supported by
the National Key R$\&$D Program of China (2017YFB0802000),
the National NSF of China (61602125, 61702124, 61862011, 61862012, 61802221),
the China Postdoctoral Science Foundation (2018M633041),
the NSF of Guangxi (2016GXNSFBA380153, 2017GXNSFAA198192),
the Guangxi Science and Technology Plan Project (AD18281065),
the Guangxi Key Laboratory of Cryptography and Information Security (GCIS201817),
the research start-up grants of Dongguan University of Technology,
the Guangdong Provincial Science and Technology Plan Projects
(2016A010101034, 2016A010101035),
the Key Research and Development Program for Guangdong Province (2019B010136001),
and the Peng Cheng Laboratory Project of Guangdong Province (PCL2018KP005 and PCL2018KP004).

\ifCLASSOPTIONcaptionsoff
  \newpage
\fi

\vspace{-24pt}
\begin{IEEEbiographynophoto}{Yanbin Zheng}
received the Ph.D. degree in Mathematics from Guangzhou University, China, in 2015.
He then joined the School of Computer Science and Engineering,
Guilin University of Electronic Technology, China.
He is currently a joint postdoctoral research fellow between
the School of Computer Science and Engineering,
South China University of Technology,
and the School of Computer Science and Technology,
Dongguan University of Technology.
His research interests include finite fields and cryptography.
\end{IEEEbiographynophoto}
\vspace{-24pt}
\begin{IEEEbiographynophoto}{Qiang Wang}
was born in China where he received B. Sc., M.Sc. degrees in Mathematics
from ShaanXi Normal University (China). He received a  M. Sc. degree in
information and System Science from Carleton University (Canada) and a
Ph.D. in Mathematics from the Memorial University of Newfoundland
(Canada).  He is currently a Professor at Carleton University
in Ottawa (Canada). His main research interests are in finite fields and
applications in coding theory, combinatorics, and cryptography.
\end{IEEEbiographynophoto}
\vspace{-24pt}
\begin{IEEEbiographynophoto}{Wenhong Wei}
received the Ph.D. degree in Computer Science from
South China University of Technology, China, in 2009.
He is currently a full professor in the School of Computer Science and Technology,
Dongguan University of Technology, China.
His research interests include discrete mathematics and computer network.

\end{IEEEbiographynophoto}


\vfill


\end{document}